\documentclass[onefignum,onetabnum]{siamonline220329}

\usepackage{lipsum}
\usepackage{amsfonts}
\usepackage{graphicx}
\usepackage{epstopdf}
\usepackage{algorithmic}
\ifpdf
  \DeclareGraphicsExtensions{.eps,.pdf,.png,.jpg}
\else
  \DeclareGraphicsExtensions{.eps}
\fi

\usepackage{amsmath}
\usepackage{mathrsfs}
\usepackage{dsfont}

\usepackage{subcaption}
\usepackage{tikz}
\usepackage{stackengine}

\usepackage{accents}

\newcommand{\ubar}[1]{\underaccent{\bar}{#1}}

\newcommand{\R}{\mathbb{R}}
\newcommand{\Pp}{\mathscr{P}}

\renewcommand{\O}{\mathcal{O}}

\newcommand{\KL}{\mathcal{KL}}
\newcommand{\e}{\varepsilon}
\newcommand{\Perp}{\mathrm{PP}}

\renewcommand{\iint}{\int\hspace{-0.285cm}\int}

\DeclareMathOperator*{\argmin}{arg\,min}

\newtheorem{prop}[theorem]{Proposition}
\newtheorem{defn}[theorem]{Definition}
\numberwithin{equation}{section}

\usepackage{enumitem}
\setlist[enumerate]{leftmargin=.5in}
\setlist[itemize]{leftmargin=.5in}

\newsiamremark{remark}{Remark}
\newsiamremark{hypothesis}{Hypothesis}
\crefname{hypothesis}{Hypothesis}{Hypotheses}
\newsiamthm{claim}{Claim}

\headers{Scaling Limits for TSNE}{Ryan Murray and Adam Pickarski}

\title{Large data limits and scaling laws for tSNE\thanks{Submitted to the editors DATE.
\funding{This work was funded by NSF DMS 2307971 and the Simons Foundation MP-TSM.}}}

\author{Ryan Murray\thanks{Department of Mathematics, North Carolina State University, Raleigh, NC, USA  (\email{rwmurray@ncsu.edu}).}
\and Adam Pickarski\thanks{Department of Mathematics, North Carolina State University, Raleigh, NC, USA (\email{appickar@ncsu.edu}.}}

\usepackage{amsopn}

\ifpdf
\hypersetup{
  pdftitle={Large Data Limits for tSNE},
  pdfauthor={Ryan Murray and Adam Pickarski}
}
\fi

\begin{document}

\maketitle

\begin{abstract}
This work considers large-data asymptotics for t-distributed stochastic neighbor embedding (tSNE), a widely-used non-linear dimension reduction algorithm. We identify an appropriate continuum limit of the tSNE objective function, which can be viewed as a combination of a kernel-based repulsion and an asymptotically-vanishing Laplacian-type regularizer. As a consequence, we show that embeddings of the original tSNE algorithm cannot have any consistent limit as $n \to \infty$. We propose a rescaled model which mitigates the asymptotic decay of the attractive energy, and which does have a consistent limit. \end{abstract}


\section*{Acknowledgments}
RM and AP thank Kevin Miller for useful discussion about early versions of this work.

\section{Introduction}
Nonlinear dimension reduction for data visualization is a critical aspect of modern data science pipelines, allowing for the interpretation of high-dimensional datasets. Arguably the most successful method in this endeavor is t-distributed Stochastic Neighbor Embedding (tSNE) \cite{van2008visualizing}.  This algorithm has been highly influential in many applied fields, for example in medical imaging \cite{jamieson2010exploring}, genomics \cite{kobak2019art}, marine biology \cite{cieslak2020t}, mechanical engineering \cite{jiang2022fault}, hyperspectral imaging \cite{melit2020unsupervised}, molecular dynamics \cite{spiwok2020time}, and agriculture \cite{luo2021visualization}.  It has optimized implementations in various standard data science libraries, and has over 47000 citations in google scholar (2024): it is clear that tSNE is a highly utilized and influential algorithm for dimension reduction.

The tSNE algorithm incorporates several different approaches to obtain desirable dimension reduction embeddings. At its core it minimizes a non-convex energy which balances between attraction and repulsion terms. It measures affinity between the original features by utilizing locally adaptive kernels, and measures the distance between embedded points using heavy tailed interactions. Finally, the algorithm utilizes a training schedule which emphasizes attraction and repulsion differently at different stages of the optimization: this is known as \emph{early exaggeration}. While tSNE performs quite well in many settings, the interaction between the various engineering decisions made in its design have made it difficult to interpret rigorously. Indeed, some work \cite{wattenberg2016how} has pointed out various pitfalls in the interpretation of the outputs of tSNE. Other work, which we outline in Section \ref{sec:related_works}, has sought to provide more interpretable theoretical foundations for specific parts of the algorithm, especially the early exaggeration scheme. The goal of this work is to provide further theoretical foundations for the optimization problem underpinning the tSNE algorithm, especially in the context of large data limits.

\section{Setup and Main Results}
\subsection{Problem Setup}
Let $\mu\in\Pp(\Omega)$ be a probability distribution on $\R^d$ with support on a bounded, $C^2$ domain $\Omega$. We assume that $\mu$ has bounded density function $\rho(x)$ with respect to the Lebesgue measure, i.e. $\mu(dx)=\rho(x)dx$ where $1/\rho^*<\rho(x)<\rho^*$ for all $x\in\Omega$.  Let $X_1,...,X_n$ be independent observations from the distribution $\mu$. Following \cite{van2008visualizing}, the tSNE algorithm models the likelihood that two observations $X_i$ $X_j$ are `neighbors', given our sample, by the formula
\begin{equation}\label{p_j|i}
p_{ij}:=\frac{p_{j|i}+p_{i|j}}{2n}\qquad p_{j|i}=\frac{\exp(-|X_i-X_j|^2/2\sigma_i^2)}{\sum_{k\ne i}\exp(-|X_i-X_k|^2/2\sigma_i^2)},
\end{equation}
where $p_{j|i}$ should be thought of as the conditional probability of observing $X_j$ as a neighbor of $X_i$. Here, $p_{j|i}$ is influenced both by the sampled data and by a bandwidth parameter $\sigma_i$ (see Section \ref{subsection:perplexity}). This results in a symmetric joint distribution $P_n=\{p_{ij}\}_{i,j=1}^n$ on $\{(X_i,X_j)\}_{ij=1}^n$ that models the probability of observing the pair $(X_i,X_j)$ from the data. 

In the lower-dimensional space, one constructs a joint distribution $Q_n=\{q_{ij}\}_{i,j=1}^n$ on $\{(y_i,y_j)\}_{ij=1}^n$ over a given configuration $y_1,...,y_n\in \R^m$ by setting
\begin{equation}\label{q_ij}
q_{ij}=\frac{(1+|y_i-y_j|^2)^{-1}}{\sum_{k\ne l}(1+|y_k-y_l|^2)^{-1}}.
\end{equation}
In this lower-dimensional representation, the kernel for the joint distribution $q_{ij}$ is chosen to be the Student t-distribution to address the `crowding problem'. The crowding problem was a challenge first identified in \cite{van2008visualizing} which essentially describes the difficulty that pairwise distances are generically very big when the intrinsic dimension $d$ is large. The t-distribution alleviates this problem by allowing a moderate distance in the high-dimensional space to be modeled by a much larger distance in the embedded space. This allows for a more dispersed representation of the data in the lower-dimensional space, a feature which contributes to the success of tSNE.

The objective of tSNE is to minimize the Kullback-Leibler divergence between the joint probability distributions $P_n$ and $Q_n$. For $n$ samples, this is formulated as: \[
\mathrm{KL}(P_n || Q_n)= \sum_{ij} p_{ij}\log\frac{p_{ij}}{q_{ij}},
\]with $p_{ii}=q_{ii}=0$. One then optimizes this energy over all possible configurations $\{y_1,...,y_n\}$. In the interest of studying large data limits, we can consider the related functional problem of minimizing this energy over some space of mappings $T:\R^d\to\R^m$:\begin{equation}
\label{discrete_functional_cost}
\mathrm{KL}_n(T)= \sum_{ij} p_{ij}\log\frac{p_{ij}}{q_{ij}(T)},\qquad q_{ij}(T):=\frac{(1+|T(X_i)-T(X_j)|^2)^{-1}}{\sum_{k\ne l}(1+|T(X_k)-T(X_l)|^2)^{-1}}.
\end{equation}
where we've abused notation since this energy, as a function of $T$, no longer has the interpretation of a KL divergence.
\subsection{tSNE as Attraction-Repulsion}\label{subsection:attraction_repulsion}
It is well-documented in the literature that tSNE may be viewed as an attraction-repulsion problem. Often this is presented in terms of the ``forces'' experienced by particles during gradient descent. We will derive a similar decomposition of the tSNE objective function in terms of an attraction and repulsion energy.

 A short derivation in the appendix shows that by defining the quantities\begin{align*}
\mathrm{A}_n[T]&:=\frac{1}{n}\sum_{i=1}^n\frac{\sum_{j=1}^n\exp\Big(\tfrac{-|X_i-X_j|^2}{2\sigma_{i}^2}\Big)\log(1+|T(X_i)-T(X_j)|^2)}{\sum_{j=1}^n\exp\Big(\tfrac{-|X_i-X_j|^2}{2\sigma_{i}^2}\Big)}\\
\mathrm{R}_n[T]&:=\log\left(\frac{1}{n^2}\sum_{i=1}^n\sum_{j=1}^n\frac{1}{1+|T(X_i)-T(X_j)|^2}\right),
\end{align*}
 we can write $\mathrm{KL}_n(T)=\mathrm{D}_n+\mathrm{A}_n[T]+\mathrm{R}_n[T]$ where $\mathrm{D}_n$ is a purely data dependent term, meaning that $\mathrm{D}_n$ is independent of $T$. This formulation implies that minimization of the energy in \eqref{discrete_functional_cost} may be viewed as competition between attractive and repulsive energies, namely \begin{equation}
\argmin_{T:\R^d\to\R^m} \mathrm{KL}_n(T)=\argmin_{T:\R^d\to\R^m}\mathrm{A}_n[T]+\mathrm{R}_n[T].
 \end{equation}Heuristically, the attractive term fosters proximity among mappings, and aims to preserve intra-cluster relationships. On the other hand, the repulsive term induces a separation between different clusters. The interplay between these two terms ideally results in a well-partitioned embedding in $\R^m$.\begin{remark}
We notice here that for a fixed map $T:\R^d\to\R^m$ and point $x\in\R^d$, the summand of the attractive term may be viewed as a locally weighted average of $\log(1+|T(x)-T(X)|^2)$ and is not unlike the Nadaraya-Watson regressor. We observe that this term is minimized as $T\to 0$.

We also notice that the repulsive term is reminiscent of a Coulomb energy, with the notable absence of a singularity when $T(X_i)=T(X_j)$. Without the singularity, it has been observed that the repulsive kernel can be represented well by low rank approximations \cite{Linderman2017EfficientAF}. It is straightforward to check that the repulsive term decreases without bound as $T\to \infty$. 
\end{remark}

\subsection{Perplexity and Bandwidth}\label{subsection:perplexity}
From now on, to make explicit the dependence on all the samples for the choice of bandwidth, we will write $\sigma_{i,n}$ for the bandwidth at $X_i$ for a dimension reduction problem of $n$ samples. In the proceeding discussion, we will also make explicit the dependence of the bandwidth in the conditional probabilities by writing $p_{j|i}(\sigma_{i,n})$.

The aforementioned bandwidths are selected to maintain a constant perplexity. Perplexity, which can be understood as a smoothed measure of the number of neighbors, is typically defined as the base 2 exponential of the Shannon entropy of the conditional distribution $p_{\cdot|i}$, measured in bits \cite{NIPS2002_6150ccc6, van2008visualizing}. Contrary to entropy, perplexity is completely independent of the base, so for mathematical convenience we will use base $e$. To this end, define the perplexity of the conditional distribution $p_{\cdot | i}$ as: \begin{equation}\label{Perp1}
\Perp_n(X_i |\sigma_{i,n}):=\exp\left(-\sum_{k\ne i}^n p_{k|i}(\sigma_{i,n})\log(p_{k|i}(\sigma_{i,n}))\right).
\end{equation}

The bandwidth $\sigma_{i,n}$ is then chosen, for every $i$, so that the perplexity $\Perp_n(X_i) = \kappa$ for all $i$. In essence, by choosing a fixed perplexity, tSNE can adaptively tune the bandwidth parameter $\sigma_{i,n}$ for each observation $X_i$ to account for the local density variations. Conceptually, this process seeks an appropriate bandwidth that ensures the desired `effective number of neighbors'. If the data points are widely dispersed around $X_i$, a larger bandwidth is required to encompass the specified number of neighbors. Moreover, as the sample size grows indefinitely, the count of observations within any finite region also increases without bound, regardless of the underlying density. This necessitates progressively smaller values of $\sigma_{i,n}$ to conserve the defined number of neighbors. 

This adaptive bandwidth strategy in tSNE to model the joint distribution $P_n$ is akin to the approach used in adaptive kernel density estimation. In the constant bandwidth case, it is well known that for a kernel $K:\R^d\to\R$, the quantity\begin{equation}\label{KDE}
\frac{1}{nh_n^d}\sum_{i=1}^n K\left(\frac{x-X_i}{h_n}\right)
\end{equation}is a pointwise consistent estimator of the density $\rho(x)$ as long as $nh_n^d/\log(n)\to\infty$.

In fact, this exact scaling of the bandwidths may be realized as a consequence from holding perplexity (almost) constant. We may see this from a brief derivation in the appendix which shows that equation \eqref{Perp1} may be rewritten as\[
\Perp_n(X_i |\sigma_{i,n})=\left(\sum_{k=1}^n\exp\left(-\frac{|X_i-X_k|^2}{2\sigma_{i,n}^2}\right)\right)\exp\left(\frac{\sum_{k=1}^n \frac{|X_i-X_k|^2}{2\sigma_{i,n}^2}\exp\Big(\frac{-|X_i-X_k|^2}{2\sigma_{i,n}^2}\Big)}{\sum_{k=1}^n\exp\Big(\frac{-|X_i-X_k|^2}{2\sigma_{i,n}^2}\Big)}\right).\] 
In this formula, we notice that we could replace $X_i$ with a generic $x$ and $\sigma_{i,n}$ with a function $\sigma(x)$. 

With an application of Hoeffding's inequality, it is straightforward to show that in order to maintain a constant perplexity, one must let $\sigma_{i,n}=\O(n^{-1/d})$ with high probability. However it is well known in the kernel density estimation literature that this particular scaling will not facilitate consistency of the estimator defined in \eqref{KDE}. In order to leverage results from the study of kernel density estimation, we instead allow perplexity to scale slowly with $n$, for example at any rate $n^{\alpha}, \alpha>0$, rather than being held exactly constant.
\subsection{Main Results}
Motivated by the previous discussion, our first proposition identifies the exact limiting bandwidth which arises as a result of allowing perplexity to grow slowly with the number of samples. This proposition simultaneously illustrates how in this perplexity regime, bandwidths shrink to zero and furthermore elucidates precisely how the bandwidth is chosen adaptively to compensate for the varying density. In what follows, we will simplify the analysis by writing $\sigma_{i,n}$ as $h_n \sigma_i$ where $h_n$ is deterministic and $\sigma_i$ encompasses all of the randomness associated with the draw $X_1,....,X_n$.
\begin{prop}[Local Adaptivity]\label{Scaling_prop}Suppose that $h_n$ is a sequence for which $nh_n^d/\log(n)\to\infty$ and $h_n\to 0$. Suppose further that $\rho(x)$ is a uniformly continuous density that is bounded above and below. Then if $\widehat\sigma_{n}(x)$ is chosen so that $\Perp_n(x|h_n\widehat\sigma_n(x))=\kappa nh_n^d$, then for $\tilde\Omega\subset\subset\Omega$ then $n \to \infty$ \[\lim_{n\to\infty}\left\|\widehat\sigma_n(x)-\frac{1}{\sqrt{2\pi e}}\left(\frac{\kappa}{\rho(x)}\right)^{1/d}\right\|_{L^\infty(\tilde \Omega)}=0,\qquad a.s.\]
\end{prop}
In the sequel, we let $\sigma_\kappa(x)$ denote the limiting object above.
The proof of this proposition follows directly from standard results in kernel density estimation. The specific coefficient appearing provides the correct renormalization when using \eqref{Perp1} to estimate the density.

This result does not permit for a constant perplexity, as proposed in the original algorithm. However, it allows for any perplexity which grows faster than $\log(n)$ and slower than $n$. 
Here the coefficient $\kappa$ serves as the parameter representing, after rescaling, the ``effective number of neighbors," as discussed earlier.

Having identified the limit $\sigma_\kappa$, in order to simplify the proofs in the remainder of our results we will be considering the energy having replaced the stochastic, implicitly defined $\sigma_{i,n}$ with the deterministic $h_n \sigma_\kappa(x)$. After making this simplification, for any fixed $T$, we now consider an averaged version of our energies:
\begin{align*}
\tilde {\mathrm{A}}_h[T]&:=\int_\Omega\frac{\int_\Omega\exp\Big(\tfrac{-|x-x'|^2}{2h^2\sigma^2_{\kappa}(x)}\Big)\log(1+|T(x)-T(x')|^2)\rho(x')dx'}{\int_\Omega\exp\Big(\tfrac{-|x-x'|^2}{2h^2\sigma^2_{\kappa}(x)}\Big)\rho(x')dx'}\rho(x)dx\\
\tilde {\mathrm{R}}[T]:&=\log\left(\iint_{\Omega\times \Omega} \frac{1}{1+|T(x)-T(x')|^2}\rho(x)\rho(x')dxdx'\right).
\end{align*}
We briefly remark that while the above quantities are not the expected values of the stochastic counterparts $\mathrm{A}_n[T]$ and $\mathrm{R}_n[T]$, by the law of large numbers and the continuous mapping theorem, for fixed bandwidth, one has $\mathbb{E}[\mathrm{A}_n[T]]\to \tilde {\mathrm{A}}_h[T]$ and $\mathbb{E}[\mathrm{R}_n[T]]\to \tilde {\mathrm{R}}[T]$ as $n\to \infty$. 

For these energies, the fact that the bandwidth goes to zero (i.e. Proposition \ref{Scaling_prop}) gives, in the authors' opinion, a rather surprising result. In particular, we can see that for a wide range of function spaces, this averaged tSNE energy will not have a limiting solution, which comprises our first main result.
\begin{theorem}[Ill-posedness]\label{lack_of_existence}
Let $T_n$ be a sequence of minimizers of the energies\[
  \tilde {\mathrm{A}}_{h_n}[T]+\tilde{\mathrm{R}}[T].
\]Then $T_n$ does not converge pointwise to any $T^*\in L^\infty(\Omega;\R^m)$.
\end{theorem}
This behavior can be explained by viewing tSNE from the lens of a force-based method that balances attractive and repulsive forces. Roughly speaking, the idea behind the proof is that as the number of samples grows the bandwidth of the kernel will shrink. The attractive term is dominated by a function of the bandwidth and consequently will shrink to zero. This means that as $n$ grows the repulsive term dominates causing solutions to expand unboundedly. While we have proven this result for the averaged (deterministic) energy, we suspect that the same result would hold for the original stochastic optimization problem.

This insight provides an additional explanation for a well-established engineering technique in tSNE, namely \emph{early exaggeration}. In practice, early exaggeration involves premultiplying the attractive term by a constant greater than one during early stages of gradient descent, which has been empirically shown to yield better results. Our observation clarifies that this practice is not merely heuristic but is, in fact, asymptotically consistent with the underlying behavior of the algorithm. As the attractive force weakens with increasing sample size, early exaggeration could potentially compensate for this effect.

The empirical evidence of the need for early exaggeration to stabilize the algorithm, as well as the asymptotic lack of minimizers, strongly suggests the need to introduce a new model which circumvents these problems yet maintains most of the meaningful structure. To address this, we propose a rescaled version of tSNE, where the entire attractive force is premultiplied by $1/h_n^2$. Incorporating this, we may write the rescaled model for $n$ samples as a minimization of the following energy:\begin{equation}\label{rescaled_discrete}
  \widehat{\mathrm{KL}}_n(T)=\frac{\mathrm{A}_n[T]}{h_n^2}+\mathrm{R}_n[T].
\end{equation}
Notice that we have now transitioned back to the stochastic energy defined in Section \ref{subsection:attraction_repulsion} however here we will be using $\sigma_{i,n}=h_n\sigma_\kappa(X_i)$ as derived in Proposition \ref{Scaling_prop}.
Our second main result establishes the convergence of this energy, for fixed $T$ (i.e. pointwise convergence), toward the following limiting energy: 
\begin{align}\label{continuum_energy} 
&\KL(T):=\nonumber\\
&\frac{\kappa^{2/d}}{2\pi e}\int_\Omega \sum_{i=1}^m |\nabla T_i(x)|^2\rho^{1-2/d}(x)dx+\log\left(\iint_{\Omega\times\Omega} \frac{1}{1+|T(x)-T(x')|^2}\rho(x)\rho(x')dxdx'\right),
\end{align}
\begin{theorem}[Consistency]\label{Main_result}
Let $\mu$ be a distribution on a compact set $\Omega\subset \R^d$ with Lebesgue density bounded above and below. Then for every $T\in C^2(\Omega;\R^m)$ we have\[
\lim_{n\to\infty} \widehat{\mathrm{KL}}_n(T)\to\KL(T).
\]
\end{theorem}
In fact, we show something slightly stronger than the above statement by deriving exact probabilistic rates for convergence, which can be seen in Proposition \ref{prop:probabilistic_bounds}.

\begin{remark}
We notice that by defining the modified pairwise affinities in the embedded space by \[
\tilde q_{ij}(T)=\frac{(1+|T(X_i)-T(X_j)|^2)^{-h_n^{-2}}}{\sum_{k\ne l}(1+|T(X_k)-T(X_l)|^2)^{-1}}.
\]we can write \[\widehat{\mathrm{KL}}_n(T)=\sum_{ij}p_{ij}\log\frac{p_{ij}}{\tilde q_{ij}(T)}\]which can be viewed as a pseudo Kullback-Leibler divergence as $\tilde q_{ij}(T)$ no longer forms a probability distribution.
\end{remark}

Having identified a candidate for the limiting variational problem, we then demonstrate that this problem has a well-behaved minimizer.
\begin{theorem}[Well-posedness]
  Let $\mu$ be a distribution supported on a bounded, connected, $C^1$ domain $\Omega\subset \R^d$ with Lebesgue density $\rho(x)$ bounded above and below on $\Omega$. Then there exists $T^*\in H^1(\Omega,\rho)$ for which \[
    \KL(T^*)=\inf_{T:\Omega\to\R^m} \KL(T).
\]
\end{theorem}
This theorem stands in contrast to the lack of minimizers of the original, unscaled, tSNE energy. While we do not prove it here, the result of Theorem \ref{Main_result} strongly suggests that minimizers of the rescaled discrete energy \eqref{rescaled_discrete} ought to converge to minimizers of the limiting continuum energy \eqref{continuum_energy}. A rigorous treatment of this question, usually described under the topic of $\Gamma$-convergence, is beyond the scope of this paper. We do remark that some of the assumptions of this theorem (regarding non-degenerate densities and connected domain) are necessary, and one can construct counterexamples without them.

While these results are interesting in their own right, the rescaled model offers more than just a theoretical advancement—it opens up new avenues for understanding the fundamental mechanics of tSNE. By addressing the inherent issues in the traditional model, such as the diminishing attractive force and lack of compactness in minimizers, our new approach provides a more stable and consistent framework. This enhanced understanding not only deepens our knowledge of tSNE’s underlying principles but also paves the way for developing more robust algorithms that avoid the pitfalls of the original model. These innovations could lead to improved performance in practical applications, offering solutions that are both theoretically sound and practically effective.

One immediate consequence of the rescaled model, is that the necessary conditions are given by a system of elliptic PDEs (one for each embedded dimension) with a well-behaved right hand side. Of course, this has implications about the regularity of optimal embeddings $T$, which are easily shown to be smooth. We summarize this in the following theorem.
\begin{theorem}[Necessary Conditions and Regularity]\label{regularity} Any minimizer $T^*$ of $\KL$ will satisfy (in the weak sense) the equation
\begin{equation}\label{necessary_condition}
  \frac{\kappa^{2/d}}{2\pi e}\nabla\cdot \big(\nabla T^*_\ell(x)\rho^{1-2/d}(x)\big)=-4\frac{\int_\Omega \big(T^*_\ell(x)-T^*_\ell(x')\big)\big(1+|T^*(x)-T^*(x')|^2\big)^{-2}\rho(x')dx'}{\int\hspace{-0.2cm}\int_{\Omega\times\Omega} \big(1+|T^*(x)-T^*(x')|^2\big)^{-1}\rho(x)\rho(x')dxdx'},
\end{equation}
with a Neumann condition in each component of $T^*$ (i.e. $\nabla T_\ell^*(x) \cdot \nu = 0$ for each $x \in \partial \Omega$ where $\nu$ is the normal vector to $\partial \Omega$). Furthermore, if $\rho \in C^{k,\alpha}(\Omega)$ for $k \geq 1$ then $T^* \in C^{k+1,\alpha}(\Omega ; \R^m)$.
\end{theorem}

\section{Related Works}\label{sec:related_works}
In recent years, tSNE has garnered considerable attention due to its success in visualizing high-dimensional data. Some of the earlier theoretical research on this algorithm was directed to understanding its propensity to separate well-clustered data. For example, \cite{linderman2017clustering} utilized a dynamical systems approach to model the particle system arising from the embedding. They derived a relationship between the early exaggeration parameter and the step size in the gradient descent which caused the embedded clusters to shrink in diameter. Building on this, the authors in \cite{DBLP:journals/corr/abs-1803-01768} gave the first theoretical guarantees on the quality of the visualizations produced by tSNE. In particular, by formalizing the notion of visualizing clustered data, they argue that the results in \cite{linderman2017clustering} could not rule out the possibility that distinct clusters collapsing into each other in the embedding. Of particular technical relevance to our approach is that \cite{DBLP:journals/corr/abs-1803-01768} chooses to replace the perplexity-defined bandwidths with a more explicit rule, namely
\[
\sigma_i^2 = \frac{\gamma}{4}\cdot \min_{j\ne i} |X_i-X_j|^2
\]where $\gamma$ is a parameter determined by the shape and separation of the clusters. It is well known that ``closest point'' distances of this form scale on the order of $n^{-1/d}$ which agrees with the scaling observed in our work.

Several other theoretical works have focused on studying tSNE in terms of attraction-repulsion dynamics. The balance between attraction and repulsion forces was first identified in the original tSNE paper \cite{van2008visualizing}, and was further investigated in \cite{Forcefulcolorings,9929350}. For example, \cite{Forcefulcolorings} identifies the magnitude of the forces acting on a particular particle as a feature encoding additional useful information. We build upon this work by examining the scaling relationship between the attraction and the repulsion as the number of data points increases.


When viewed in terms of attraction and repulsion, early exaggeration can be understood simply as an increased affect of the attractive term by multiplying by a scalar $\alpha>1$. Continuing in this perspective, \cite{JMLR:v23:21-0055} perform an empirical analysis of effect of varying $\alpha$. They demonstrate that for stronger attractive forces embeddings can better represent continuous manifold structures, while for stronger repulsive forces embeddings can better recover discrete cluster structure. Notably, they provide additional empirical evidence to a question initially posed by  \cite{linderman2017clustering}: \textit{is the early exaggeration stage of tSNE equivalent to power iterations in spectral embeddings?} This question remained unanswered until the recent work of  \cite{cai2022theoretical} which established the asymptotic equivalence of the early exaggeration stage with power iterations from Laplacian eigenmaps in a rigorous way amongst other significant contributions. In particular, they argue that for strongly clustered data one can simply replace the early exaggeration stage by a spectral embedding using the eigenvectors corresponding to the $m$ smallest eigenvectors of the Laplacian matrix associated to the matrix $P_n$, and then to continue afterwards using gradient descent on the original tSNE energy.

In all of these works, the number of particles was taken to be finite, and hence the limiting behavior of the model remained unexplained. 
The only work that we are aware of in this direction is the recent work \cite{auffinger2023equilibriumdistributionstdistributedstochastic}. They establish a large data limit for a modified version of tSNE, and demonstrate convergence towards a non-local, kernel-based energy similar to what we call $\tilde A_h$ and $\tilde R$. The large data result in \cite{auffinger2023equilibriumdistributionstdistributedstochastic} is stated in terms of limits of minimizers, which is notably stronger than our result. However, they require the perplexity to grow at a rate proportional to $n$. Consequently their limiting model does not localize to a Dirichlet energy and stays completely nonlocal. Our work complements theirs by allowing the bandwidths to decrease to zero and requiring only very mild growth of the perplexity (e.g. like $\log(n)$), and along the way we identify a loss of compactness of minimizers which does not occur in the case where perplexity scales like $n$. 


\section{Proof of Main Results}
\subsection{Stochastic Convergence of the Bandwidths}
The starting point for our proofs is to first rigorously establish that the bandwidth goes to zero, and to precisely quantify the local adaptivity. To begin, we will present a simplified version of Corollary 1 from \cite{Einmahl_and_Mason} which is a kernel density estimation consistency result holding uniformly across bandwidths. 
\begin{lemma}\label{Einmahl_and_Mason}Assume $K$ is (i) bounded, (ii) $K(s)=\phi(p(s))$ for a polynomial $p:\R^d\to \R$ and $\phi$ right continuous, (iii) \[
\int_{\R^d} \sup_{|y|\geq |x|} |K(y)|dx<\infty,
\]and (iv) integrates to 1. Then for any sequences $0 < a_n < b_n < 1$, satisfying $b_n \to 0$ and $n a_n/ \log n \to\infty$, and any uniformly continuous density $\rho$, we have\[
\lim_{n\to\infty}\sup_{a_n\leq h\leq b_n}\left\|\frac{1}{nh^d}\sum_{i=1}^n K\left(\frac{x-X_i}{h}\right)-\rho(x)\right\|_\infty=0\qquad a.s.
\]
\end{lemma}

The authors in \cite{Einmahl_and_Mason} note that this result is ``immediately applicable to proving uniform consistency of kernel-type estimators when the bandwidth $h$ is a function of the location $x$...", a fact which is immediately pertinent in our setting. The main idea of the proof of Proposition \ref{Scaling_prop} will be to apply this result with specific choices of $a_n,b_n$. We do remark here that we this result will not hold on the entirety of $\Omega$ due to the presence of the boundary. However, for any compactly embedded subset $\tilde\Omega\subset\subset \Omega$ the above convergence result will hold.
\begin{proof}[Proof of Proposition \ref{Scaling_prop}]
First note that by writing \[
K_1(s):=\frac{1}{(2\pi)^{d/2}}\exp\left(\frac{-|s|^2}{2}\right),\qquad K_2(s):=\frac{1}{d(2\pi)^{d/2}}|s|^2\exp\left(\frac{-|s|^2}{2}\right),
\]we have \[
\frac{\Perp_n(x|\sigma)}{n\sigma^d}=(2\pi)^{d/2}\frac{1}{n\sigma^d}\sum_{i=1}^nK_1\left(\frac{x-X_i}{\sigma}\right)\exp\left(\frac{d}{2}\cdot\frac{\tfrac{1}{n\sigma^d}\sum_{i=1}^nK_2\left(\frac{x-X_i}{\sigma}\right)}{\tfrac{1}{n\sigma^d}\sum_{i=1}^nK_1\left(\frac{x-X_i}{\sigma}\right)}\right)
\]
By Lemma \ref{Einmahl_and_Mason} along with the probabilistic continuous mapping theorem (along with the strict positivity of the density) we then have that for sequences $0<a_n<b_n<1$ for which $b_n\to 0$ and $n a_n^d/\log(n)\to\infty$
\[\lim_{n\to\infty}\sup_{a_n\leq\sigma\leq b_n}\left\|\frac{\Perp_n(x|\sigma)}{n\sigma^d}-(2\pi e)^{d/2}\rho(x)\right\|_{L^\infty(\tilde \Omega)}=0.\]In particular, for fixed $\delta>0$, if $h_n$ is as in the statement, by writing $a_n=c_1 h_n$ and $b_n=c_2 h_n$ with\[
c_1=\frac{1}{\sqrt{2\pi e}}\left(\frac{\kappa-\delta}{ \rho^*}\right)^{1/d},\qquad c_2=\frac{1}{\sqrt{2\pi e}}\left(\frac{\kappa+\delta}{ {\rho^*}^{-1}}\right)^{1/d}
\] and $\sigma=h_n \sigma(x)$, we can rewrite the above result as \begin{equation}\label{convergence_of_perplexity}
\lim_{n\to\infty}\sup_{c_1\leq \sigma(x)\leq c_2}\left\|\frac{\Perp_n(x|h_n\sigma(x))}{nh_n^d}-(2\pi e)^{d/2}\sigma^d(x)\rho(x)\right\|_{L^\infty(\tilde \Omega)}=0,
\end{equation}
where we are free to multiply the terms by $\sigma(x)$ since it is bounded. Furthermore for fixed $x\in\R^d$, given $\eta>0$ for $n$ sufficiently large,\begin{align*}
\mathbb{P}\left(\left|\frac{\Perp_n(x|h_nc_2)}{nh_n^d}-\frac{(\kappa+\delta)}{\ubar \rho}\cdot \rho(x)\right|<\frac{\delta}{2}\right)=1-\eta
\end{align*}
In particular with probability at least $1-\eta$,\[
\kappa+\delta/2<\frac{\Perp_n(x|h_nc_2)}{nh_n^d}.
\]
The monotonicity of $\sigma\mapsto\Perp(x|\sigma)$ thus implies $\widehat \sigma_n(x)\leq c_2$. Doing the same for the lower bound, we can conclude $n$ sufficiently large $
\mathbb{P}(\widehat\sigma_n(x)\in[c_1,c_2])>1-2\eta,
$ in particular,\[
\lim_{n\to\infty}\mathbb{P}(\widehat\sigma_n(x)\in[c_1,c_2])=1.
\]Furthermore, we can assume without loss of generality that the events \[
\{\widehat\sigma_n(x)\in[c_1,c_2]\},\qquad\{\widehat\sigma_{n+1}(x)\in[c_1,c_2]\}
\]are independent, by considering independent draws of the $X_i$ for each $n$.
Thus by the Borel-Cantelli lemma, we can pass the limit inside the probability operator to conclude\[
\mathbb{P}(\limsup_{n\to\infty}\{\widehat\sigma_n(x)\in[c_1,c_2]\})=1.
\]In particular, this shows that the tail of the sequence $\widehat\sigma_n(x)$ is uniformly bounded almost surely in $n$. Therefore, we may apply Equation \eqref{convergence_of_perplexity} to the tail of the sequence $\widehat \sigma_n(x)$ to show
 \[
\lim_{n\to\infty}\left\|\kappa-(2\pi e)^{d/2}\widehat\sigma_n^d(x)\rho(x)\right\|_{L^\infty(\tilde \Omega)}=0,
\]
which implies the result.
\end{proof}

\subsection{Non-existence of Limiting Minimizers}
We now turn our attention to the ``ill-posedness'' of the tSNE energy, in the sense that minimizers of the averaged energy cannot converge as $n \to \infty$. To begin, we present the following function class which will be useful in facilitating our proof.
\begin{defn}\label{omega_class}Let $\omega:\R^+\to\R^+$ be a continuous, increasing function with $\omega(0)=0$. We say that $T\in F_\omega(\Omega;\R^m)$ if \[
    \int_\Omega \frac{\int_\Omega \exp\Big(\tfrac{-|x-x'|^2}{2h^2\sigma^2_{\kappa}(x)}\Big)|T(x)-T(x')|^2\rho(x)\rho(x')dx'}{\int_\Omega \exp\Big(\tfrac{-|x-x'|^2}{2h^2\sigma^2_{\kappa}(x)}\Big)\rho(x')dx'} \,dx \leq \omega(h).
\]
\end{defn}
In Lemma \ref{lem:F_omega_big}, we will show that actually every function in $L^\infty(\Omega;\R^m)$ is an element of $F_\omega(\Omega;\R^m)$ for some $\omega$. However for some function classes one can directly construct the function $\omega$: for example for the H\"older space $C^\alpha(\Omega ; \R^m)$ it suffices to consider $\omega(h) = C h^{2\alpha}$. This definition mainly serves to provide a convenient framework for us to prove the following lemma.
\begin{lemma}
Let $T_n$ be a sequence of minimizers of the energies\[
  \tilde {\mathrm{A}}_{h_n}[T]+\tilde{\mathrm{R}}[T],
\] with $h_n \to 0$. Then $T_n$ does not pointwise converge to any $T^*\in F_\omega(\Omega;\R^m)$.
\end{lemma}
\begin{proof}Suppose for the sake of contradiction that there exists a limiting map $T^*\in F_\omega(\Omega;\R^m)$, and consider the sequence of maps $\{T^*/\big(\omega(h_n)\big)^{\alpha/2}\}_{n=1}^\infty$ for any $\alpha\in (0,1)$. Notice that \begin{align*}
    \tilde {\mathrm{A}}_{h_n}\left[\frac{T^*}{\big(\omega(h_n)\big)^{\alpha/2}}\right]&\leq \frac{1}{\big(\omega(h_n)\big)^{\alpha}}\int_\Omega \frac{\int_\Omega \exp\Big(\tfrac{-|x-x'|^2}{2h^2\sigma^2_{\kappa}(x)}\Big)|T^*(x)-T^*(x')|^2\rho(x)\rho(x')dx'}{\int\exp\Big(\tfrac{-|x-x'|^2}{2h^2\sigma^2_{\kappa}(x)}\Big)\rho(x')dx'} \,dx\\
&\leq \big(\omega(h_n)\big)^{1-\alpha}
\end{align*}
since $T^*\in F_\omega(\Omega;\R^m)$ by assumption. Furthermore, since $\tilde {\mathrm{A}}_{h_n}[T]\geq 0$ for every $T$, we have\[
\tilde {\mathrm{A}}_{h_n}\left[\frac{T^*}{\big(\omega(h_n)\big)^{\alpha/2}}\right]\leq \tilde {\mathrm{A}}_{h_n}[T_n]+ \big(\omega(h_n)\big)^{1-\alpha}.
\]Next, we estimate the contribution due to the repulsion when $\omega(h_n)<1:$\begin{align*}
\tilde{\mathrm{R}}\left[\frac{T^*}{\big(\omega(h_n)\big)^{\alpha/2}}\right]&=\log\left(\iint \frac{\big(\omega(h_n)\big)^{\alpha}}{\big(\omega(h_n)\big)^{\alpha}+|T^*(x)-T^*(x')|^2}\rho(x)\rho(x')dxdx'\right)\\
&\leq \alpha\log\big(\omega(h_n)\big)+\tilde{\mathrm{R}}[T^*].
\end{align*}The pointwise convergence of $T_n\to T^*$ as well as the fact that the integrand of $\tilde{\mathrm{R}}[T]$ is bounded by 1 implies, by the the dominated convergence theorem, that for every $\varepsilon>0$ there is $n$ sufficiently large for which $\tilde{\mathrm{R}}[T^*]<\tilde{\mathrm{R}}[T_n]+\varepsilon$. Additionally, since $\big(\omega(h_n)\big)^{1-\alpha}+\alpha\log\big(\omega(h_n)\big)\to-\infty$ as $h_n\to 0$, we may bound this quantity above by $-2\varepsilon$ by choosing $n$ large enough. Together, these considerations show that \[
\tilde {\mathrm{A}}_{h_n}\left[\frac{T^*}{\big(\omega(h_n)\big)^{\alpha/2}}\right]+\tilde{\mathrm{R}}\left[\frac{T^*}{\big(\omega(h_n)\big)^{\alpha/2}}\right]<\tilde{ \mathrm{A}}_{h_n}[T_n]+\tilde{\mathrm{R}}[T_n]-\varepsilon
\]which contradicts the minimality of $T_n$.
\end{proof} We now justify our earlier claim that the function space $F_\omega(\Omega;\R^m)$ is quite broad. 
\begin{lemma}\label{lem:F_omega_big}
  If $T \in L^\infty(\Omega;\R^m)$ then $T \in F_\omega(\Omega;\R^m)$ for some continuous, increasing $\omega:\R^+ \to \R^+$ with $\omega(0) = 0$.
  \label{lem:T_in_F_omega}
\end{lemma}

\begin{proof}
  By Lusin's theorem, and letting $\rho(K_\e^c) := \int_{\Omega \cap K_\e^c} \rho(x) \,dx$, for every $\e>0$ there exists a compact set $K_\e$ so that $\rho(K_\e^c) \leq \e$ and so that restricted to $K_\e$ we have that $T$ is continuous. As $T$ is continuous on $K_\e$, it is uniformly continuous on $K_\e$, and hence there exists a continuous, increasing function $\bar \omega_\e:\R^+ \to \R^+$ so that $|T(x) - T(x')| \leq \bar \omega_\e(|x-x'|)$ for all $x,x' \in K_\e$, and $\bar \omega_\e(0) = 0$. Defining $Q_h$ to be the set in the product space so that $|x-x'| < \sqrt{h}$, we may then estimate
  \begin{align*}
    \tilde{\mathrm{A}}_h[T] &\leq \big(\bar \omega_\e(\sqrt{h})\big)^2 \iint_{(K_\e\times K_\e) \cap Q_h} \frac{\exp\Big(\frac{-|x-x'|^2}{2h^2 \sigma_\kappa^2(x)}\Big)}{\int\exp\Big(\frac{-|x-x'|^2}{2h^2 \sigma_\kappa^2(x)}\Big) \rho(x') dx' } \rho(x)\rho(x') dx dx' \\
      &+ C\|T\|_\infty \iint_{(K_\e\times K_\e)^c} \frac{\exp\Big(\frac{-|x-x'|^2}{2h^2 \sigma_\kappa^2(x)}\Big)}{\int\exp\Big(\frac{-|x-x'|^2}{2h^2 \sigma_\kappa^2(x)}\Big) \rho(x') dx' } \rho(x)\rho(x') dx dx' \\
	  &+  C\|T\|_\infty \iint_{(K_\e\times K_\e) \cap Q_h^c} \frac{\exp\Big(\frac{-|x-x'|^2}{2h^2 \sigma_\kappa^2(x)}\Big)}{\int\exp\Big(\frac{-|x-x'|^2}{2h^2 \sigma_\kappa^2(x)}\Big) \rho(x') dx' } \rho(x)\rho(x') dx dx' \label{eqn:L1-est} 
	    \end{align*}
	    The first term can be readily bounded by $\big(\bar \omega_\e(\sqrt{h})\big)^2 $, whereas the last term we may bound by $C\|T\|_\infty h^{-d}e^{-h^{-1}}$ since $\rho$ is bounded from below. The second term we can split into two integrals, one over $\Omega \times K_\e^c$ and the other over $K_\e^c \times \Omega$. We then notice that
\[
	      \int_{\Omega} \int_{K_\e^c} \frac{\exp\Big(\frac{-|x-x'|^2}{2h^2 \sigma_\kappa^2(x)}\Big)}{\int_\Omega\exp\Big(\frac{-|x-x'|^2}{2h^2 \sigma_\kappa^2(x)}\Big) \rho(x') dx' } \rho(x)\rho(x') dxdx' = \int_{K_\e^c} \rho(x) dx \leq \e.
\]
	    For the other integral, we compute
\[
	      \int_{K_\e^c} \int_{\Omega} \frac{\exp\Big(\frac{-|x-x'|^2}{2h^2 \sigma_\kappa^2(x)}\Big)}{\int\exp\Big(\frac{-|x-x'|^2}{2h^2 \sigma_\kappa^2(x)}\Big)\rho(x') dx' } \rho(x)\rho(x') dx dx'= \int_{K_\e^c} \hspace{-0.2cm}\rho(x')\hspace{-0.1cm}\int_{\Omega} \frac{1}{h^d} \exp\Big(\frac{-|x-x'|^2}{2h^2 \sigma_\kappa^2(x)}\Big )g(x) dx dx'
\]
where
\[
  g(x) = \frac{\rho(x)}{\tfrac{1}{h^d} \int_\Omega \exp\Big(\frac{-|x-x'|^2}{2h^2 \sigma_\kappa^2(x)}\Big) \,dx'}.
\]
As the denominator converges uniformly in $x$ to a multiple of the numerator, and both are bounded away from zero, we have that $g$ is bounded. By then integrating, we find that the last term is also bounded by a constant times $\e$.

Putting these facts together, we then have that there exists a $\tilde h$ so that for all $h < \tilde h$ we have that $\tilde{\mathrm{A}}_h[T] < C\e$, which completes the proof.
\end{proof}
Combining these two lemmas, we have proven Theorem \ref{lack_of_existence}.
\subsection{Consistency of the Rescaled Model}
The main goal of this section is to prove Theorem \ref{Main_result}. We first note by the law of large numbers the repulsion term converges to the integrated repulsion term in $\mathcal{KL}$, and hence all of the work in this section will go towards proving convergence of the attractive term. The main idea of the proof is relatively standard: we utilize Taylor approximations along with concentration inequalities to estimate tail probabilities. The analysis in many ways is analogous to the approach used to prove consistency for manifold learning and spectral embeddings, see for example \cite{6789755}.  

Before stating those quantified convergence results, we will first define a few quantities. To begin, we point out that by setting\[
U_{ij} :=\frac{1}{h_n^{d+2}} \exp\left(\tfrac{-|X_i-X_j|^2}{2h_n^2\sigma_\kappa^2(X_i)}\right)\log(1+|T(X_i)-T(X_j)|^2),\quad V_{ij}:=\frac{1}{h_n^d}\exp\left(\tfrac{-|X_i-X_j|^2}{2h_n^2\sigma_\kappa^2(X_i)}\right).
\]we may express the attractive term in $\widehat{\mathrm{KL}}_n(T)$ as\[
\frac{\mathrm{A}_n[T]}{h_n^2}=\frac{1}{n}\sum_{i=1}^n\frac{\sum_{j=1}^n U_{ij}}{\sum_{j=1}^n V_{ij}}.
\]From this we observe that by changing variables  
\begin{align} 
\mathbb{E}[U_{ij}|X_i]&=\sigma_\kappa^{d+2}(X_i)\int_{\frac{\Omega-X_i}{h_n}}\exp\left(\tfrac{-|v|^2}{2}\right)|DT(X_i)\cdot v|^2\rho(X_i+h_nv)dv+\O(h_n)\nonumber\\
&=(2\pi)^{d/2}\sigma_\kappa^{2+d}(X_i)\rho(X_i)\sum_{\ell=1}^m|\nabla T_\ell(X_i)|^2+\O(h_n), \label{eqn:uij-Taylor}\\
\mathbb{E}[V_{ij}|X_i]&=\sigma_\kappa^d(X_i)\int_{\frac{\Omega-X_i}{h_n}}\exp\left(\tfrac{-|v|^2}{2}\right)\rho(X_i+h_nv)dv\nonumber\\
&=(2\pi)^{d/2}\sigma_\kappa^d(X_i)\rho(X_i)+\O(h_n).\label{eqn:vij-Taylor}
\end{align}for and $X_i$ in the interior of $\Omega$. Here the order $h_n$ term is uniform for points which are uniformly separated from the boundary of $\Omega$. Additionally, it will be important in the proof to ensure that $\mathbb{E}[V_{ij}|X_i]$ is bounded from below by a constant independent of $h_n$, independent of $X_i$. 
This can be bounded from below directly, as we have
\begin{equation} \label{eqn:lower-bound-vij}
  \mathbb{E}[V_{ij}|X=x]\geq\frac{\kappa^{2/d}}{2\pi e(\rho^*)^{2/d}}\frac{1}{\rho^*}\inf_{x\in \Omega}\int_{\frac{\Omega-x}{h_n}}e^{-|v|^2/2}dv=:\frac{1}{\tilde \rho},
\end{equation}
where this bound is independent of $h_n$ since the volume of the region $(\Omega-x)/h_n$ increases as $h_n\to 0$. We also note that $\mathbb{E}[U_{ij}|X_i]$ can be bounded above since
\begin{equation}\label{eqn:upper-bound-uij}
\mathbb{E}[U_{ij}|X_i]\leq(2\pi)^{d/2}\sup_{x\in \Omega}\left[\sigma_\kappa^{2+d}(x)\rho(x)\sum_{\ell=1}^m|\nabla T_\ell(x)|^2\right]+C=:\tilde \sigma
\end{equation}
as $\nabla T_i$ is bounded on $\bar \Omega$. Lastly, we also remark that the set
\[
\left\{\left|\frac{\sum_{j=1}^n U_{ij}}{\sum_{j=1}^n V_{ij}}-\frac{\mathbb{E}[U_{ij}|X_i]}{\mathbb{E}[V_{ij}|X_i]}\right|>\frac{\varepsilon}{2}\right\}
\]
can directly be shown to be a subset of 
\begin{equation}\label{eqn:subset-bound}
 \left\{\left|\frac{1}{n}\sum_{j=1}^n U_{ij}-\mathbb{E}[U_{ij}|X_i]\right|>\frac{\varepsilon\tilde \rho^{-1}}{8}\right\}\cup\left\{\left|\frac{1}{n}\sum_{j=1}^n V_{ij}-\mathbb{E}[V_{ij}|X_i]\right|> \min\left\{\frac{\tilde \rho^{-1}}{2},\frac{\varepsilon\tilde\rho^{-2}}{8\tilde \sigma}\right\}\right\},
\end{equation}
where we have used \eqref{eqn:lower-bound-vij}. To simplify our notation, we define
\[
  \eta_{\varepsilon}:=\frac{\varepsilon\tilde\rho^{-1}}{8},\qquad \xi_{\varepsilon}:=\min\left\{\frac{\tilde \rho^{-1}}{2},\frac{\varepsilon\tilde\rho^{-2}}{8\tilde \sigma}\right\}.
\]
We now state the following proposition, which directly implies Theorem \ref{Main_result}. 
\begin{prop}\label{prop:probabilistic_bounds}Let $T\in C^2(\bar \Omega;\R^m)$, where $\Omega$ is a bounded, $C^2$ domain and $\rho$ is a probability density bounded above and below on $\Omega$ . Then for every $\varepsilon>0$, and $0 < \alpha < 1$, there exist constants $C,\tilde C$ (only depending upon $\Omega$, $\rho$ and $\|T\|_{C^2}$) such that with probability at least \[
1-2\exp\left(\frac{-n\varepsilon^2}{C(1+\varepsilon/3)}\right)-2n\exp\left(\frac{-nh_n^d\eta_{\varepsilon}^2}{C(1+\eta_{\varepsilon}/3)}\right)-2n\exp\left(\frac{-nh_n^d\xi_{\varepsilon}^2}{C(1+\xi_{\varepsilon}/3)}\right),
\]
we have that
\[
  \left|\frac{\mathrm{A}_n[T]}{h_n^2}-\int_\Omega\sum_{\ell=1}^m \sigma_\kappa^2(x)|\nabla T_\ell(x)|^2\rho(x)dx\right|\leq \varepsilon+\tilde C h_n^\alpha.
\]In particular if $h_n\to 0$ and $nh_n^d/\log n\to \infty$, we have almost surely $\widehat{\mathrm{KL}}_n(T)\to\KL(T)$ for each fixed $T\in C^2(\bar\Omega;\R^m).$\end{prop}
\begin{proof}
  To begin, we estimate
\begin{align}
  &\mathbb{P}\left(\left|\frac{\mathrm{A}_n[T]}{h_n^2}-\int_\Omega\sum_{\ell=1}^m \sigma_\kappa^2(x)|\nabla T_\ell(x)|^2\rho(x)dx\right|> \varepsilon+C_4h_n\right) \label{eqn:prob-est-main}\\
&\leq \mathbb{P}\left(\left|\frac{1}{n}\sum_{i=1}^n\frac{\sum_{j=1}^n U_{ij}}{\sum_{j=1}^n V_{ij}}-\sum_{i=1}^n\frac{\mathbb{E}[U_{ij}|X_i]}{\mathbb{E}[V_{ij}|X_i]}\right|> \varepsilon/2\right)\\
&+\mathbb{P}\left(\left|\frac{1}{n}\sum_{i=1}^n\frac{\mathbb{E}[U_{ij}|X_i]}{\mathbb{E}[V_{ij}|X_i]}- \int_\Omega \frac{\mathbb{E}[U_{ij}|x]}{\mathbb{E}[V_{ij}|x]}\rho(x) dx \right|> \varepsilon/2 \right)\\
&+\mathbb{P}\left(\left|\int_\Omega\frac{\mathbb{E}[U_{ij}|x]}{\mathbb{E}[V_{ij}|x]} \rho(x) dx-\int_\Omega\sum_{\ell=1}^m \sigma_\kappa^2(x)|\nabla T_\ell(x)|^2\rho(x)dx\right|> C_4 h_n^\alpha\right).
\end{align}
The last term here is actually deterministic, so we just must show that it is smaller than $C h_n^\alpha$ for some $C$. This can be directly justified by removing a $h_n^\alpha$ width tube from around the boundary of both integrals: the contribution from this tube can be bounded by a constant times $h_h^\alpha$ as we have that $\mathbb{E}[V_{ij}|x]$ and $\mathbb{E}[U_{ij}|x]$ are respectively bounded from below and above for any point in the closure of $\Omega$, along with the fact that $\Omega$ is a $C^2$ domain. The difference of the integrals in the region well-separated from the boundary can then be directly estimated to be of size $h_n$ by again using the upper and lower bounds along with the Taylor expansions \eqref{eqn:uij-Taylor} and \eqref{eqn:vij-Taylor}, which will hold uniformly in any region which is much more than distance $h_n$ from the boundary: for this reason it is convenient to use the power $\alpha$.

To bound the first term in \eqref{eqn:prob-est-main}, by \eqref{eqn:subset-bound} we may bound
\begin{align}\label{quotient_bound}
\mathbb{P}\left(\left|\frac{\sum_{j=1}^n U_{ij}}{\sum_{j=1}^n V_{ij}}-\frac{\mathbb{E}[U_{ij}|X_i]}{\mathbb{E}[V_{ij}|X_i]}\right|>\varepsilon/2\right)&\leq \mathbb{P}\left(\left|\frac{1}{n}\sum_{j=1}^n U_{ij}-\mathbb{E}[U_{ij}|X_i]\right|>\eta_{\varepsilon}\right)\nonumber\\
&+\mathbb{P}\left(\left|\frac{1}{n}\sum_{j=1}^n V_{ij}-\mathbb{E}[V_{ij}|X_i]\right|>\xi_{\varepsilon}\right).
\end{align}
In the following, we let the constant $C$ vary from line to line for convenience. To estimate the above quantities we note that for $h_n<1$, we estimate with Taylor expansion \begin{align*}
U_{ij}&=\frac{1}{h_n^{d+2}} \exp\left(\tfrac{-|X_i-X_j|^2}{2h_n^2\sigma^2(X_i)}\right)\log(1+|T(X_i)-T(X_i+(X_j-X_i))|^2)\\
&\leq\frac{1}{h_n^d}\exp\left(\tfrac{-|X_i-X_j|^2}{2h_n^2\sigma^2(X_i)}\right)|DT(X_i)|^2\cdot\frac{|X_j-X_i|^2}{h_n^2}+C h_n^{1-d} \\
&\leq\frac{2}{eh_n^d}\|\sigma\|_\infty^2 \|DT\|_\infty^2+Ch_n^{1-d}\leq Ch_n^{-d}\end{align*}and furthermore,\begin{align*}
\mathrm{Var}(U_{ij})&\leq \mathbb{E}[U_{ij}^2]\\
&=\frac{1}{h_n^{2d+4}}\int \exp\left(\tfrac{-|X_i-\tilde x|^2}{h_n^2\sigma^2(X_i)}\right)\big(\log(1+|T(X_i)-T(\tilde x)|^2)\big)^2\rho(\tilde x)d\tilde x\\
&=\frac{1}{h_n^d}\int\exp\left(-|v|^2\right)|DT(X_i)\cdot v|^4\rho(X_i+h_nv)dv+\O(h_n^{1-d})\leq C h_n^{-d}\end{align*}
Similarly, we can show $V_{ij}\leq h_n^{-d}$ and $\mathrm{Var}(V_{ij})\leq C h_n^{-d}.$ Thus with \eqref{quotient_bound} and Bernstein's inequality, we see that \begin{align*}
\mathbb{P}\left(\left|\frac{\sum_{j=1}^n U_{ij}}{\sum_{j=1}^n V_{ij}}-\frac{\mathbb{E}[U_{ij}|X_i]}{\mathbb{E}[V_{ij}|X_i]}\right|>\varepsilon/2\right)\leq 2\exp\left(\frac{-nh_n^d\eta_{\varepsilon,\delta}^2}{C(1+\eta_{\varepsilon,\delta}/3)}\right)+2\exp\left(\frac{-nh_n^d\xi_{\varepsilon,\delta}^2}{C(1+\xi_{\varepsilon,\delta}/3)}\right),
\end{align*}
and furthermore a union bound will imply \[
  \mathbb{P}\left(\left|\frac{\mathrm{A}_n[T]}{h_n^2}-\frac{1}{n}\sum_{i=1}^n\frac{\mathbb{E}[U_{ij}|X_i]}{\mathbb{E}[V_{ij}|X_i]}\right|>\frac{\varepsilon}{2}\right)\leq 2n\exp\left(\frac{-nh_n^d\eta_{\varepsilon,\delta}^2}{C(1+\eta_{\varepsilon,\delta}/3)}\right)+2n\exp\left(\frac{-nh_n^d\xi_{\varepsilon,\delta}^2}{C(1+\xi_{\varepsilon,\delta}/3)}\right).\]
Finally, for the middle term in \eqref{eqn:prob-est-main} we notice that by \eqref{eqn:lower-bound-vij} and \eqref{eqn:upper-bound-uij} we can directly apply Bernstein's inequality to obtain that
\[
  \left|\frac{1}{n}\sum_{i=1}^n\frac{\mathbb{E}[U_{ij}|X_i]}{\mathbb{E}[V_{ij}|X_i]}- \int_\Omega  \frac{\mathbb{E}[U_{ij}|x]}{\mathbb{E}[V_{ij}|x]}\rho(x)dx \right| > \frac{\varepsilon}{2}
\]
with probability at most
\[
  2 \exp\left( -\frac{n \e^2}{8(\frac{\tilde \sigma^2}{\tilde \rho^2} + 1/3\frac{\tilde \sigma}{\tilde \rho} \e )} \right).
\]
This concludes the proof.
\end{proof}

\subsection{Minimizers of rescaled tSNE}

We now give a brief proof of the existence of minimizers to the limiting, rescaled, population-level energy: the minimizer of this energy should be a candidate for the large-data limit of our rescaled variant of tSNE. The proofs in this section are standard in the Calculus of variations, see for example the standard reference works \cite{dacorogna2007direct,giusti2003direct} which give a more detailed treatment.

Because our densities are bounded from above and below on $\Omega$, it suffices for us to consider the unweighted function spaces $L^2(\Omega;\R^m)$ and $H^1(\Omega;\R^m)$, which are defined by the space of measurable functions from $\Omega \to \R^m$ equipped with the norms
\begin{displaymath}
  \|T\|_{L^2}^2 := \int_\Omega |T(x)|^2 \,dx \qquad \|T\|_{H^1}^2 := \|T\|_{L^2}^2 +  \int_\Omega \sum_{\ell=1}^m |\nabla T_\ell(x)|^2 \,dx.
\end{displaymath}

We recall the classical Sobolev embedding theorem and Rellich-Kondrakov compactness theorem:

\begin{theorem}[Poincar\'e-Wirtinger Inequality]
  Let $\Omega$ be a $C^1$ smooth, bounded, connected domain in $R^d$. Then there exists a constant $C$ such that for any function $u:\Omega \to \R$ such that $\int_\Omega |\nabla u|^2 \,dx<\infty$ we have that
  \begin{displaymath}
    \left\|u-\int_\Omega u \,dx \right\|_2 \leq C\|\nabla u\|_2.
  \end{displaymath}
  \label{thm:sobolev-embedding}
\end{theorem}

\begin{theorem}[Rellich-Kondrachov Theorem]
  Suppose that $\|u_n\|_{H^1} < C$. Then there exists a subsequence (not relabeled) and a limit $u^* \in H^1$ so that $u_n \to u^*$ in $L^2$ and so that 
  \begin{equation}
    \liminf \int_\Omega |\nabla u_n|^2\,dx \geq \int_\Omega |\nabla u^*|^2\,dx.
  \end{equation}
  \label{thm:R-K}
\end{theorem}

We can now combine these results to obtain the existence of minimizers of our limiting energy.

\begin{theorem}
  Let $\Omega$ be a $C^1$ bounded, connected domain in $\R^d$ and consider the energy \eqref{continuum_energy}, namely
  \begin{align*}
    \KL[T] &= \frac{\kappa^{2/d}}{2\pi e} \int_\Omega \sum_{\ell=1}^m |\nabla T_\ell(x)|^2 \rho(x)\,dx + \log\left(\iint_{\Omega\times\Omega} \frac{1}{1 + |T(x) - T(x')|^2}  \rho(x) \rho(x') \,dx \,dx'\right) \\
    &=: \mathrm{ A}[T] + \mathrm{ R}[T] .
  \end{align*}
  Then there exists a minimizer $T^* \in H^1(\Omega ; \R^m)$ of the energy $\KL$.
  \label{thm:min-exist-rescaled-energy}
\end{theorem}

\begin{proof}
  We first notice that both $\mathrm{ A}$ and $\mathrm{ R}$ are translation invariant. Furthermore, by using Jensen's inequality and the fact that, for any $\beta > 0$ we have $\log(1+x) \leq \beta^{-1} + \log(\beta/e) + x/\beta$, we can estimate
  \begin{displaymath}
    \mathrm{R}[T] \geq C_\beta + \frac{C}{\beta} \int |T(x)|^2 \,dx.
  \end{displaymath}
  Using our Poincar\'e inequality we can then infer that, after centering by subtracting $\int_\Omega T$ and picking $\beta$ sufficiently small,
  \[
    \KL[T] \geq C \|T\|_{H^1(\Omega; \R^m)} - C.
  \]
  We also notice that $\KL[0] < \infty$. This implies that there exists a sequence of centered $T_n$ so that $\KL[T_n] \to \inf_{H^1} \KL$ and $\|T_n\|_{H^1(\Omega;\R^m)} < C$. The Rellich-Kondrachov theorem then implies that $T_n \to T^*$ in $L^2(\Omega)$ for some $T^* \in H^1(\Omega; \R^m)$ and that $\mathrm{ A}[T^*] \leq \liminf_n \mathrm{ A}[T_n]$. The dominated convergence theorem along with the $L^2$ convergence of the $T_n$ implies that $\mathrm{ R}[T_n] \to \mathrm{ R}[T^*]$. These facts together imply that $\KL[T^*] \leq \liminf_n \KL[T_n]$, which proves that $T^*$ is a minimizer.
\end{proof}

Finally, by considering $\KL[T^* + t \phi]$ for arbitrary $\phi$ and taking $t \to 0$ we can immediately infer the following weak form of the Euler-Lagrange equation:

\begin{proposition}
  Any minimizer $T^*$ of $\mathcal{KL}$ will satisfy
  \begin{align*}
    0 &= \frac{\kappa^{2/d}}{2\pi e} \int_\Omega \sum_{\ell=1}^m \nabla T^*_\ell(x) \cdot \nabla \phi_\ell(x) \rho^{1-2/d}(x) dx \\
    &- \frac{\int_\Omega \int_\Omega 2(\phi(x) - \phi(x')) \cdot (T^*(x) - T^*(x')) (1 + |T^*(x) - T^*(x')|^2)^{-2}\rho(x) \rho(x')dx dx'}{\int_\Omega \int_\Omega (1 + |T^*(x) - T^*(x')|)^{-1} \rho(x) \rho(x')dx dx'}
      \end{align*}
  for any $\phi \in H^1(\Omega ; \R^m)$.
\end{proposition}

As $\phi$ is arbitrary, we can integrate by parts to obtain the equality \eqref{necessary_condition}, in the sense of $H^1$ functions. In order to prove Theorem \ref{regularity}, it only remains to resolve the regularity of $T^*$: this follows from a standard argument in elliptic regularity theory. In doing so, we notice that the right hand side of our necessary condition is globally Lipschitz in $T^*(x)$, and hence any regularity established for $T^*$ also holds for the right hand side of the equation. Using elliptic regularity theory, we can then ``bootstrap'' to infer greater regularity upon $T^*$, a process that we can continue for as many derivatives as $\rho$ possesses. We remark that this regularity stands in contrast to recent results for other nonlinear dimension reduction methods such as one standard variant of multi-dimensional scaling \cite{murray2024probabilistic}: the ramifications of such a regularity result are still a topic of investigation.

\section{Appendix}
\subsection*{Derivation of Perplexity}
Noting equation \eqref{Perp1} and plugging in \eqref{p_j|i}, one finds
\begin{align*}
&\log(\Perp_n(X_i|\sigma))=-\sum_{j} p_{j|i}\log p_{j|i}\\
&=-\sum_{j} \frac{\exp(-|X_i-X_j|^2/2\sigma^2(X_i))}{\sum_{k}\exp(-|X_i-X_k|^2/2\sigma^2(X_i))}\Bigg[-\frac{|X_i-X_j|^2}{2\sigma^2(X_i)}-\log\left(\sum_{k}\exp\left(-\frac{|X_i-X_k|^2}{2\sigma^2(X_i)}\right)\right)\Bigg]\\
&=\frac{\sum_{j} |X_i-X_j|^2/2\sigma^2(X_i)\exp(-|X_i-X_j|^2/2\sigma^2(X_i))}{\sum_{k}\exp(-|X_i-X_k|^2/2\sigma^2(X_i))}+\log\left(\sum_{k}\exp\left(-\frac{|X_i-X_k|^2}{2\sigma^2(X_i)}\right)\right)\\
&\implies \Perp_n(X_i|\sigma)=\left(\sum_{k}\exp\left(-\frac{|X_i-X_k|^2}{2\sigma^2(X_i)}\right)\right)\exp\left(\frac{\sum_{k} |X_i-X_k|^2\exp\left(-\frac{|X_i-X_k|^2}{2\sigma^2(X_i)}\right)}{2\sigma^2(X_i)\sum_{k}\exp\left(-\frac{|X_i-X_k|^2}{2\sigma^2(X_i)}\right)}\right)
\end{align*}
\subsection*{Derivation of Population Model}

Expanding equation \eqref{discrete_functional_cost} we see
\begin{align*}
\mathrm{KL}_n(T)&=\sum_{ij}p_{ij}\log p_{ij}+\sum_{ij}p_{ij}\log(1+|T(X_i)-T(X_j)|^2)\\
&+\sum_{ij}p_{ij}\log\left(\sum_{kl}(1+|T(X_k)-T(X_l)|^2)^{-1}\right).
\end{align*}
Using equation \eqref{p_j|i} and noting the symmetry of $p_{ij}$, one has
\begin{align*}
\mathrm{KL}_n(T)&=\frac{1}{n}\sum_{ij}p_{j|i}\log \left(\frac{p_{j|i}+p_{i|j}}{2n}\right)+\frac{1}{n}\sum_{ij}p_{j|i}\log(1+|T(X_i)-T(X_j)|^2)\\
&+\log\left(\sum_{kl}(1+|T(X_k)-T(X_l)|^2)^{-1}\right).
\end{align*}since $\sum_{ij} p_{ij}=1$ in the final term. Plugging in the form of $p_{j|i}$ defined in  \eqref{p_j|i}, we find
\begin{align*}
&\mathrm{KL}_n(T)\\
&=\frac{1}{n}\sum_{ij}\frac{\exp\Big(\tfrac{-|X_i-X_j|^2}{2\sigma_{i,n}^2}\Big)}{\sum_{k}\exp\Big(\tfrac{-|X_i-X_k|^2}{2\sigma_{i,n}^2}\Big)}\log\left(\frac{1}{2n}\frac{\exp\Big(\tfrac{-|X_i-X_j|^2}{2\sigma_{i,n}^2}\Big)}{\sum_{k}\exp\Big(\tfrac{-|X_i-X_k|^2}{2\sigma_{i,n}^2}\Big)}+\frac{1}{2n}\frac{\exp\Big(\tfrac{-|X_j-X_i|^2}{2\sigma_{j,n}^2}\Big)}{\sum_{k}\exp\Big(\tfrac{-|X_j-X_k|^2}{2\sigma_{j,n}^2}\Big)}\right)\\
&+\frac{1}{n}\sum_{ij}\frac{\exp\Big(\tfrac{-|X_i-X_j|^2}{2\sigma_{i,n}^2}\Big)\log(1+|T(X_i)-T(X_j)|^2)}{\sum_{k}\exp\Big(\tfrac{-|X_i-X_k|^2}{2\sigma_{i,n}^2}\Big)}+\log\left(\sum_{ij}(1+|T(X_i)-T(X_j)|^2)^{-1}\right).
\end{align*}
To obtain a consistent limit as $n\to \infty$ for $\{\sigma_{i,n}\}_{i=1}^n$ which is constant order in $n$, we manipulate further to get 
\begin{align*}
&=\frac{1}{n^2}\sum_{ij}\frac{\exp\Big(\tfrac{-|X_i-X_j|^2}{2\sigma_{i,n}^2}\Big)}{\tfrac{1}{n}\sum_{k}\exp\Big(\tfrac{-|X_i-X_k|^2}{2\sigma_{i,n}^2}\Big)}\Bigg[\log\left(\frac{\tfrac{1}{2}\exp\Big(\tfrac{-|X_i-X_j|^2}{2\sigma_{i,n}^2}\Big)}{\tfrac{1}{n}\sum_{k}\exp\Big(\tfrac{-|X_i-X_k|^2}{2\sigma_{i,n}^2}\Big)}+\frac{\tfrac{1}{2}\exp\Big(\tfrac{-|X_j-X_i|^2}{2\sigma_{j,n}^2}\Big)}{\tfrac{1}{n}\sum_{k}\exp\Big(\tfrac{-|X_j-X_k|^2}{2\sigma_{j,n}^2}\Big)}\right)\\
&-2\log(n)\Bigg]
\\
&+\frac{1}{n}\sum_{ij}\frac{\exp\Big(\tfrac{-|X_i-X_j|^2}{2\sigma_{i,n}^2}\Big)\log(1+|T(X_i)-T(X_j)|^2)}{\sum_{k}\exp\Big(\tfrac{-|X_i-X_k|^2}{2\sigma_{i,n}^2}\Big)}\\
&+\log\left(\frac{1}{n^2}\sum_{ij}(1+|T(X_i)-T(X_j)|^2)^{-1}\right)+2\log(n).
\end{align*}Notice here that the $\log(n)$ terms exactly cancel since $\sum_{ij}p_{ij}=1$. Furthermore, the very first term has no dependence on $T$ and therefore can be dropped in the minimization over $T$.

\bibliographystyle{siamplain}
\bibliography{TSNE-bib}

\end{document}